\documentclass[11pt]{article}
\usepackage{paper}

\usepackage{relsize}
\usepackage{tikz,pgfplots}
\usepackage[mathscr]{euscript}

\title{A Statistical Learning Assessment of Huber~Regression}

\author[1]{Yunlong Feng}
\author[2]{Qiang Wu}
\affil[1]{Department of Mathematics and Statistics, University at Albany}
\affil[2]{Department of Mathematical Sciences, Middle Tennessee State University}

\date{}

\begin{document}
\maketitle
    
\begin{abstract}
\noindent As one of the triumphs and milestones of robust statistics, Huber regression plays an important role in robust inference and estimation. It has also been finding a great variety of applications in machine learning. In a parametric setup, it has been extensively studied. However, in the statistical learning context where a function is typically learned in a nonparametric way, there is still a lack of theoretical understanding of how Huber regression estimators learn the conditional mean function and why it works in the absence of light-tailed noise assumptions. To address these fundamental questions, this paper conducts an assessment of Huber regression from a statistical learning viewpoint. First, we show that the usual risk consistency property of Huber regression estimators, which is usually pursued in machine learning, cannot guarantee their learnability in mean regression. Second, we argue that Huber regression should be implemented in an adaptive way  to perform mean regression, implying that one needs to tune the scale parameter in accordance with the sample size and the moment condition of the noise. Third, with an adaptive choice of the scale parameter, we demonstrate that Huber regression estimators can be asymptotic mean regression calibrated under $(1+\epsilon)$-moment conditions ($\epsilon>0$) on the conditional distribution. Last but not least, under the same moment conditions, we establish almost sure convergence rates for Huber regression estimators. Note that the $(1+\epsilon)$-moment conditions accommodate the special case where the response variable possesses infinite variance and so the established convergence rates justify the robustness feature of Huber regression estimators. In the above senses, the present study provides a systematic statistical learning assessment of Huber regression estimators and justifies their merits in terms of robustness from a theoretical viewpoint. 
\end{abstract}

\section{Introduction and Motivation}
In this paper, we are concerned with the robust regression problem where one aims at seeking a functional relation between input and output when the response variable may be heavy-tailed \cite{huber2009robust,rousseeuw2005robust,maronna2006robust,hampel2011robust}. In such scenarios, the traditionally frequently used least squares regression paradigms may not work well due to the amplification of the least squares loss to large residuals. As an alternative, the Huber loss was proposed in the seminal work \cite{huber1964robust} in the context of robust estimation of location parameters. The Huber loss and the theoretical findings in location parameter estimation then applied and carried over to robust regression problems. The regression paradigm that is associated with the Huber loss is termed as \textit{Huber regression} and the resulting estimator is termed as the \textit{Huber regression estimator}. The introduction of Huber regression led to the development of various subsequent M-estimators and fostered the development of robust statistics into a discipline.  

Denoting $X$ as the input variable that takes values in a compact metric space $\mathcal{X}\subset\mathbb{R}^d$ and $Y$ the response variable taking values in $\mathcal{Y}\subset\mathbb{R}$, given i.i.d observations $\mathbf{z}=\{(x_i,y_i)\}_{i=1}^n$, in the context of parametric regression, the Huber regression estimator $f_{\mathbf{z},\sigma}(X)=X^\top \hat{\beta}$ is learned from the following empirical risk minimization (ERM) scheme 
\begin{align}\label{Huber_ERM}
f_{\mathbf{z},\sigma}:=\arg\min_{f\in\mathcal{H}}\frac{1}{n}\sum_{i=1}^n \ell_\sigma(y_i-f(x_i)),
\end{align} 
where $\mathcal{H}$ is the function space from $\mathcal{X}$ to $\mathbb{R}$ consisting of linear functions of the form $f(x)=x^\top\beta$ and $\ell_\sigma$ is the well-known Huber loss defined by
\begin{align}\label{huber_loss} 
\ell_\sigma(t)=
\begin{cases}
t^2, &\quad \hbox{if}\,\, |t|\leq \sigma, \\
2\sigma |t| - \sigma^2,&\quad \hbox{otherwise}.
\end{cases}
\end{align}
Assuming that the conditional mean function $f^\star(X)=\mathbb{E}(Y|X)$ can be parametrically represented as  $f^\star(X)=X^\top\beta^\star$ and the noise $Y-f^\star(X)$ is zero mean when conditioned on $X$, asymptotic properties of $\hat{\beta}$ and its convergence to $\beta^\star$ have been extensively studied in the literature of parametric statistics. An incomplete list of related literature includes \cite{huber1973robust,yohai1979asymptotic,portnoy1984asymptotic,he1996general,maronna2006robust,huber2009robust,rousseeuw2005robust,hampel2011robust,loh2017statistical} and many references therein. Note that in the aforementioned studies, the scale parameter $\sigma$ in the Huber loss is set to be fixed and chosen according to the $95\%$ asymptotic efficiency rule. In a high-dimensional setting, Huber regression with a fixed scale parameter, however, may not be able to learn $\beta^\star$ when the noise is asymmetric, as argued recently in \cite{sun2019adaptive,fan2019adaptive}. There the authors proposed to choose the scale parameter by relating it to the dimension of the input space, the moment condition of the noise distribution, and the sample size so that one may debias the resulting regression estimator and the scale parameter can play a trade-off role between bias and robustness. 

In a nonparametric statistical learning context where functions in $\mathcal{H}$, in general, do not admit parametric representations, theoretical investigations of Huber regression estimators are still sparse though they have been applied extensively into various applications where robustness is a concern. To proceed with our discussion, denote $\mathcal{H}$ as a compact subset of the space $C(\mathcal{X})$ of continuous functions on $\mathcal{X}$, $\rho$ the underlying unknown distribution over $\mathcal{X}\times\mathcal{Y}$, and $\mathcal{R}^\sigma(f)$ the \textit{generalization error} of $f:\mathcal{X}\rightarrow\mathbb{R}$ defined by
\begin{align*}
\mathcal{R}^\sigma(f)=\mathbb{E}\ell_\sigma(Y-f(X)),
\end{align*} 
    where the expectation is taken jointly with respect to $X$ and $Y$. Recall that the objective is to learn the conditional mean function $f^\star(X)=\mathbb{E}(Y|X)$ robustly. Existing studies in the literature of statistical learning theory remind that Huber regression estimators are $\mathcal{R}^\sigma$-risk consistent, i.e., $\mathcal{R}^\sigma(f_{\mathbf{z},\sigma})\rightarrow \min_{f\in\mathcal H} \mathcal{R}^\sigma(f)$ as $n\rightarrow \infty$. To see this, one simply notes that the Huber loss is Lipschitz continuous on $\mathbb{R}$ and so the usual learning theory arguments may apply \cite{bartlett2006empirical,friedman2001elements,rosasco2004loss,christmann2007consistency, cucker2007learning}. However, such a risk consistency property neither says anything about the role that the scale parameter $\sigma$ plays nor implies the convergence of the regression estimator $f_{\mathbf{z},\sigma}$ to the mean regression function $f^\star$.
    
In the present study, we aim to conduct a statistical learning assessment of the Huber regression estimator $f_{\mathbf{z},\sigma}$. More specifically, we pursue answers to the following fundamental questions:
\begin{itemize}
	\item \textbf{Question 1}: Whether $\mathcal{R}^\sigma$-risk consistency implies the convergence of $f_{\mathbf{z},\sigma}$ to $f^\star$?   
	\item \textbf{Question 2}: What is the role that $\sigma$ plays when learning $f^\star$ through ERM \eqref{Huber_ERM}?
	\item \textbf{Question 3}: How to develop exponential-type fast convergence rates of $f_{\mathbf{z},\sigma}$ to $f^\star$?
	\item \textbf{Question 4}: How to justify the learnability of  $f_{\mathbf{z},\sigma}$ in the absence of light-tail noise?
\end{itemize}

Answers to these questions represent our main contributions. In particular, if $\mathcal{R}^\sigma$-risk consistency implies the convergence of $f_{\mathbf{z},\sigma}$ to $f^\star$, we say that Huber regression \eqref{Huber_ERM} is mean regression calibrated. We show that Huber regression is generally not mean regression calibrated for any fixed scale parameter $\sigma.$ Instead, it should be implemented in an adaptive way in order to perform mean regression, where the adaptiveness refers to the dependence of the scale parameter on the sample size and the moment condition. We also show that the scale parameter needs to diverge in accordance with the sample size to ensure that the Huber regression estimator $f_{\mathbf{z},\sigma}$ learns the mean regression function $f^\star$, which we term as the asymptotic mean regression calibration property. Furthermore, such an asymptotic mean regression calibration property can be  established under $(1+\epsilon)$-moment conditions ($\epsilon>0$) on the conditional distribution. This is a rather weak condition as it admits the case where the conditional distribution possesses infinite variance. To develop fast exponential-type convergence rates, we establish a relaxed Bernstein condition. The idea is to bound the second moment of associated random variables by using their first moment and an additional bias term that diminishes towards $0$ when the sample size tends to infinity. These preparations allow us to establish fast exponential-type convergence rates for $f_{\mathbf{z},\sigma}$. Interestingly, but not surprisingly, it is shown that $\sigma$ plays a trade-off role between bias and learnability, and the convergence rates of $f_{\mathbf{z},\sigma}$ depend on the order of the imposed moment conditions.

The rest of this paper is organized as follows. In Section~\ref{sec::risk_consistency_insufficient}, we argue that risk consistency is insufficient in guaranteeing learnability and so does not necessarily imply the convergence of Huber regression estimators to the mean regression function. In Section~\ref{sec::mean_calibration}, we demonstrate that Huber regression is asymptotically mean regression calibrated if the scale parameter is chosen in a diverging manner in accordance to the sample size and the moment condition. Some efforts are then made in Section~\ref{sec::realxing_Bernstein} to develop fast exponential-type convergence rates by relaxing the standard Bernstein condition in learning theory. In Section~\ref{sec::convergence_rates}, we establish fast convergence rates for Huber regression estimators under weak moment conditions. Proofs of Theorems are collected in Section~\ref{sec::proofs}. The paper is concluded in Section~\ref{sec:conclusion}.

\smallskip 
\noindent \textbf{Notation and Convention}. Throughout this paper, we assume that  $f^\star$ is bounded and $\mathcal H\subset C(\mathcal X)$ is uniformly bounded and   denote $M=\max\{\|f^\star\|_\infty,\sup_{f\in\mathcal{H}}\|f\|_\infty\}.$ Denoting  $\rho_\mathsmaller{\mathcal{X}}$ as the marginal distribution of $\rho$ on $\mathcal{X}$, then $\|\cdot\|_{2,\rho}$ defines the $L_2$-norm induced by $\rho_\mathsmaller{\mathcal{X}}$. The notation $a \lesssim b$ denotes the fact that there exists an absolute positive constant $c$ such that $a\leq c b$. For any $t\in\mathbb R$, let $t_+=\max(0, t).$ All proofs are deferred to the appendix.

\section{Risk Consistency is Insufficient in Guaranteeing Learnability}\label{sec::risk_consistency_insufficient}
In this section, we shall make efforts to answer Question 1 listed in the introduction by arguing that risk consistency is insufficient in guaranteeing learnability of the Huber regression estimator, where by \textit{risk consistency} we refer to the convergence of $\mathcal{R}^\sigma(f_{\mathbf{z},\sigma})$ to $\min_{f\in\mathcal H} \mathcal{R}^\sigma(f)$ while \textit{learnability} refers to the convergence of $f_{\mathbf{z},\sigma}$ to $f^\star$.

Following existing studies on empirical risk minimization schemes induced by convex loss functions, it is easy to deduce that $f_{\mathbf{z},\sigma}$ is $\mathcal{R}^\sigma$-risk consistent for any fixed $\sigma$ value. Moreover, under certain mild assumptions, probabilistic convergence rates may also be established. To see this, note that the deduction of the risk consistency property of Huber regression estimators as well as their convergence rates involves the following set of random variables 
\begin{align*}
{\mathcal{G}}_\mathcal{H}=\Big\{\xi_f\,|\, \xi_f:=\ell_\sigma(y-f(x))-\ell_\sigma(y-f_{\mathcal H, \sigma} (x)),\,f\in\mathcal{H}, (x,y)\in\mathcal{X}\times\mathcal{Y}\Big\},
\end{align*}
where $f_{\mathcal H,\sigma} = \arg\min_{f\in\mathcal H} \mathcal R^\sigma(f)$ is the population version of $f_{\mathbf{z},\sigma}.$ Notice that the Huber loss $\ell_\sigma$ in \eqref{huber_loss}  is Lipschitz continuous on $\mathbb{R}$ with Lipschitz constant $2\sigma$. Therefore, the random variables in ${\mathcal{G}}_\mathcal{H}$ and their variances can be  uniformly upper bounded by constants involving $\sigma$. Applying learning theory arguments and concentration inequalities to $\mathcal{G}_\mathcal{H}$, under mild assumptions, convergence rates can be derived. However, due to the dependence of $f_{\mathbf{z},\sigma}$ on the scale parameter $\sigma$, it may possess much flexibility and can be quite different with different choices of the $\sigma$ values. Consequently, the $\mathcal{R}^\sigma$-risk consistency property as well as the convergence rates of $\mathcal{R}^\sigma(f_{\mathbf{z},\sigma})-\min_{f\in\mathcal H} \mathcal{R}^\sigma(f)$ may not be informative and may not indicate the learnability of  $f_{\mathbf{z},\sigma}$ in learning $f^\star$ even if $\mathcal H$ is perfectly chosen such that  $f^\star\in\mathcal H$.   

\begin{figure}[ht]  
	\tikzset{trim left=0.0 cm}
	\input{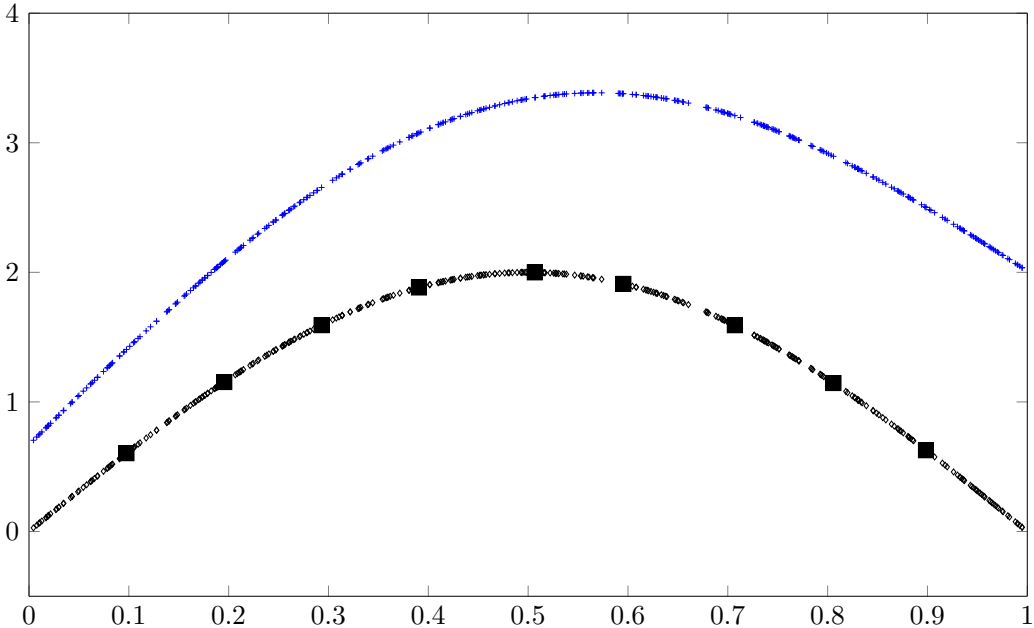}
	\caption{The bottom black curve with square marks gives the conditional mean function. The top blue curve represents the learned Huber regression estimator with %sigma value 
		$\sigma=0.01$.}\label{toy_illustration_sigma}
\end{figure} 

To illustrate  this phenomenon numerically, consider a toy example with the model $$Y=2\sin(\pi X)+(1+2X)\varepsilon,$$ where $X$ follows a uniform distribution on $[0,1]$ and $\varepsilon\sim 0.5N(0,2.5^2)+0.5N(0,0.5^2)$. It is apparent that the noise distribution admits zero mean and is skewed. Simple calculation shows that for this regression model, the conditional mean function  $f^\star(X)=2\sin(\pi X)$. In this experiment, we visualize the function $f_{\mathbf{z},\sigma}$ and compare it with $f^\star$. We choose the hypothesis space $\mathcal{H}$ as a ball of the reproducing kernel Hilbert space associated with the Gaussian kernel $K(x_i,x_j)=\exp\{-\|x_i-x_j\|^2/h^2\}$. Both the kernel bandwidth $h$ and the radius of the ball are tuned via cross-validation under the least absolute deviation error criterion while the scale parameter $\sigma$ in the Huber loss is set to be fixed with $\sigma=0.01$. A set of independent observations are sampled from the above regression model and are used as the training data. Then $f_{\mathbf{z},\sigma}$ is plotted in Figure \ref{toy_illustration_sigma}. The conditional mean function is also plotted for comparison. As discussed earlier, due to the Lipschitz continuity of the Huber loss, with the choice of $\sigma=0.01$, the risk consistency can be guaranteed. However, from the plots in Figure \ref{toy_illustration_sigma}, clearly, $f_{\mathbf{z},\sigma}$ does not approach the conditional mean function. 

The fact that $\mathcal{R}^\sigma$-risk consistency of Huber regression estimators cannot guarantee their ability to learn the conditional mean function can be further justified through the following example. Let $\mathcal{M}$ be the space of measurable functions from $\mathcal{X}$ to $\mathbb{R}$ and define
\begin{align}\label{oracle}
f_\sigma:=\arg\min_{f\in\mathcal{M}}\mathcal{R}^\sigma(f).
\end{align}  
Intuitively, $f_\sigma$ can be regarded as the best Huber regression estimator learned in an ideal case where infinite observations are available and the hypothesis space is perfectly selected. 

\begin{example}\label{prop_impossibility}
	Consider the Huber regression problem where one aims to learn the conditional mean function $f^\star$ from the following homoscedastic regression model
	\begin{align*}
	Y=f^\star(X)+\varepsilon,
	\end{align*}
	where $\varepsilon$ is the zero-mean noise variable with density 
	\begin{align*}
	p_{\varepsilon}(t)=
	\begin{cases}
	\frac{1}{2}e^{-(t+\frac{1}{4})}, & \hbox{ if } t\geq -\frac{1}{4},\\[1ex]
	e^{2(t+\frac{1}{4})}, & \hbox{ if } t<-\frac{1}{4}.
	\end{cases}
	\end{align*}
	Then there exists a constant $c$ with $c\neq 0$ such that $f_\sigma(x)=f^\star(x)+c$ for all $x\in\mathcal{X}$.
	As a result, if $f_\sigma\in\mathcal H$ and 
 $\mathcal R^\sigma(f_{\mathbf z, \sigma})$ converges to $\min_{f\in \mathcal H}\mathcal R^\sigma(f)$ as $n\to \infty$ with large probability, then $f_{\mathbf z, \sigma} $ 
 does not converge to $f^\star$ with large probability.
\end{example}

\begin{proof}
Recalling the definition of $f_\sigma$ in \eqref{oracle}, for any $x\in\mathcal{X}$, we can re-express it as follows
	\begin{align*}
	f_\sigma(x)&=\arg\min_{\nu\in\mathbb{R}}\int_{\mathbb{R}}\ell_\sigma(t-\nu)p_{\mathsmaller{Y|X=x}}(t)\mathrm{d}t\\
	&=\arg\min_{\nu\in\mathbb{R}}\int_{\mathbb{R}}\ell_\sigma(t-\nu)p_{\mathsmaller{\varepsilon}}(t-f^\star(x))\mathrm{d}t\\
	&=\arg\min_{\nu\in\mathbb{R}}\int_{\mathbb{R}}\ell_\sigma(t-\nu)p_{\mathsmaller{\varepsilon}}(t-f^\star(x))\mathrm{d}t\\
	&=\arg\min_{\nu\in\mathbb{R}}\int_{\mathbb{R}}\ell_\sigma(u-(\nu-f^\star(x)))p_{\mathsmaller{\varepsilon}}(u)\mathrm{d}u.
	\end{align*}
	Therefore, for any $x\in\mathcal{X}$, we have
	\begin{align*}
	f_\sigma(x)-f^\star(x)=\arg\min_{\nu\in\mathbb{R}}\int_{\mathbb{R}}\ell_\sigma(u-\nu)p_{\mathsmaller{\varepsilon}}(u)\mathrm{d}u.
	\end{align*} 
	The assumption that the noise $\varepsilon$ is independent of $x$ tells us that 
\begin{align}\label{equ_minimize} 
\arg\min_{\nu\in\mathbb{R}}\int_{\mathbb{R}}\ell_\sigma(u-\nu)p_{\mathsmaller{\varepsilon}}(u)\mathrm{d}u 
\end{align} 
is a unique constant for all $x\in\mathcal{X}$. To prove the first part of the assertion, we only need to show that $0$ is not a solution to the above minimization problem.

	From the definition of the Huber loss \eqref{huber_loss}, we know that
	\begin{align*}
	\int_{\mathbb{R}}\ell_\sigma(u-\nu)p_{\mathsmaller{\varepsilon}}(u)\mathrm{d}u&=\int_{\mathbb{R}}\ell_\sigma(u-\nu)p_{\mathsmaller{\varepsilon}}(u)\mathrm{d}u+\int_{\mathbb{R}}(u-\nu)^2p_{\mathsmaller{\varepsilon}}(u)\mathrm{d}u-\int_{\mathbb{R}}(u-\nu)^2p_{\mathsmaller{\varepsilon}}(u)\mathrm{d}u\\[1ex]
	&=\int_{\mathbb{R}}(u-\nu)^2p_{\mathsmaller{\varepsilon}}(u)\mathrm{d}u+\int_{|u-\nu|\geq \sigma}(2\sigma|u-v|-\sigma^2)p_{\mathsmaller{\varepsilon}}(u)\mathrm{d}u\\[1ex]
	&\hspace{3.6cm}-\int_{|u-\nu|\geq \sigma}(u-v)^2p_{\mathsmaller{\varepsilon}}(u)\mathrm{d}u\\
	&=\int_{\mathbb{R}}(u-\nu)^2p_{\mathsmaller{\varepsilon}}(u)\mathrm{d}u-\int_{|u-\nu|\geq \sigma}(|u-\nu|-\sigma)^2p_{\mathsmaller{\varepsilon}}(u)\mathrm{d}u\\[1ex]
	&=\int_{\mathbb{R}}(u-\nu)^2p_{\mathsmaller{\varepsilon}}(u)\mathrm{d}u-\int_{u-\nu\geq \sigma}(u-\nu-\sigma)^2p_{\mathsmaller{\varepsilon}}(u)\mathrm{d}u\\[1ex]
	&\hspace{3.6cm}-\int_{u-\nu\leq -\sigma}(u-\nu+\sigma)^2p_{\mathsmaller{\varepsilon}}(u)\mathrm{d}u.
	\\[1ex]
	\end{align*} 
Therefore, we have
\begin{align*}
\frac{\mathrm{d}(\int_{\mathbb{R}}\ell_\sigma(u-\nu)p_{\mathsmaller{\varepsilon}}(u)\mathrm{d}u)}{\mathrm{d}\nu}
=-2\int_{\mathbb{R}}(u-\nu)p_{\mathsmaller{\varepsilon}}(u)\mathrm{d}u
&+2 \int_{\nu+\sigma}^{+\infty}(u-\nu-\sigma) p_{\mathsmaller{\varepsilon}}(u) \mathrm{d}u\\[1ex]
& + 2 \int_{-\infty}^{\nu-\sigma} (u-\nu+\sigma)p_{\mathsmaller{\varepsilon}}(u) \mathrm{d}u.
\end{align*}
The zero-mean noise assumption tells us that 
\begin{align*}
\frac{\mathrm{d}(\int_{\mathbb{R}}\ell_\sigma(u-\nu)p_{\mathsmaller{\varepsilon}}(u)\mathrm{d}u)}{\mathrm{d}\nu}\Big |_{\nu=0}= 
2\int_{\sigma}^{+\infty}(u-\sigma)p_{\mathsmaller{\varepsilon}}(u) \mathrm{d}u 
+2\int_{-\infty}^{-\sigma}(u+\sigma)  p_{\mathsmaller{\varepsilon}}(u)\mathrm{d}u.
\end{align*}  
If $\sigma\ge \frac 1 4$, then we have
\begin{align*}
 & 2\int_{\sigma}^{+\infty}(u-\sigma)p_{\mathsmaller{\varepsilon}}(u) \mathrm{d}u 
+2\int_{-\infty}^{-\sigma}(u+\sigma)  p_{\mathsmaller{\varepsilon}}(u)\mathrm{d}u \\[1ex]
=\  &  \int_{\sigma} ^{+\infty} (u-\sigma) e^{-(u+\frac 14)} \mathrm{d}u 
+ 2\int_{-\infty}^{-\sigma} (u+\sigma) e^{2(u+\frac 14)}\mathrm{d}u \\[1ex]
=\  & e^{-\sigma-\frac{1}{4}}-\frac{1}{2}e^{\frac{1}{2}-2\sigma}>0,
\end{align*}
where we used the fact that $g(a)=a-\frac e2 a^2$ is positive for $a\in (0, \frac 2e)$ and $0<e^{-\frac 14 -\sigma} \le e^{-\frac 12} <\frac 2 e.$ If $0<\sigma< \frac 1 4$, then we have
\begin{align*} 
 & 2\int_{\sigma}^{+\infty}(u-\sigma)p_{\mathsmaller{\varepsilon}}(u) \mathrm{d}u 
+2\int_{-\infty}^{-\sigma}(u+\sigma)  p_{\mathsmaller{\varepsilon}}(u)\mathrm{d}u \\[1ex]
=\  &  \int_{\sigma} ^{+\infty} (u-\sigma) e^{-(u+\frac 14)} \mathrm{d}u 
+ \int_{-\frac 14}^{-\sigma} (u+\sigma) e^{-(u+\frac 14)} \mathrm{d}u 
+ 2\int_{-\infty}^{-\frac 14} (u+\sigma) e^{2(u+\frac 14)}\mathrm{d}u \\[1ex]
=\ & 2\sigma + e^{-\sigma-\frac{1}{4}} - e^{\sigma-\frac{1}{4}}>0,
\end{align*}
where we used the fact that the function 
$g(\sigma)= 2\sigma + e^{-\sigma-\frac{1}{4}} - e^{\sigma-\frac{1}{4}}$ has $g(0)=0$ and is strictly increasing.

Therefore, $\frac{\mathrm{d}(\int_{\mathbb{R}}\ell_\sigma(u-\nu)p_{\mathsmaller{\varepsilon}}(u)\mathrm{d}u)}{\mathrm{d}\nu}\Big |_{\nu=0}$ is non-zero for fixed $\sigma>0$, which implies that $0$ is not a solution to the minimization problem \eqref{equ_minimize}. This proves the first claim  in Example \ref{prop_impossibility}.

To prove the divergence of $f_{\mathbf{z}, \sigma}$ to $f^\star$ for any fixed $\sigma$,
first note that, by the convexity of the Huber loss, $f_\sigma$ is the unique minimizer of $\mathcal R^\sigma(f)$. Since $f_\sigma = f^\star +c \not = f^\star$, 
we have $\mathcal R^\sigma(f_\sigma) \not= \mathcal R^\sigma(f^\star).$
If $f_\sigma\in \mathcal H$ and $\mathcal R^\sigma(f_{\mathbf z,\sigma})$ 
converges to $\min_{f\in\mathcal H} \mathcal R^\sigma(f) = \mathcal R^\sigma (f_\sigma)$ with large probability,
then there exists some large enough $N$ such that 
for any $n>N$,
$$ | \mathcal R^\sigma(f_{\mathbf z,\sigma}) -
\mathcal R^\sigma (f_\sigma) | \le \frac 1 2 
\left| \mathcal R^\sigma(f_\sigma) - \mathcal R^\sigma(f^\star) \right|$$
holds with large probability. By the Lipschitz property of the Huber loss, we have 
\begin{align*}
\|f_{\mathbf z, \sigma} - f^\star\|_{2,\rho}
& \ge \frac 1 {2\sigma} \left| \mathcal R^\sigma(f_{\mathbf z, \sigma}) - \mathcal R^\sigma(f^\star)\right|  \\
& \ge \frac 1 {2\sigma } \Big( 
\left| \mathcal R^\sigma(f_\sigma) -
\mathcal R^\sigma(f^\star)  \right| 
- \left| \mathcal R^\sigma(f_{\mathbf z, \sigma}) -\mathcal R^\sigma(f_\sigma)  \right|\Big) \\
& \ge \frac 1 {4\sigma} \left| \mathcal R^\sigma(f_\sigma) -
\mathcal R^\sigma(f^\star)  \right|. 
\end{align*}
This proves that $f_{\mathbf z, \sigma}$
does not converge to $f^\star$ for any fixed $\sigma.$
\end{proof}

Example \ref{prop_impossibility} tells us that in some scenarios,  $f_{\mathbf{z},\sigma}$ may not converge to the conditional mean function which one aims to learn even if infinite samples are given and the hypothesis space is perfectly chosen. This is due to the inherent bias brought by the integrated scale parameter in the Huber loss when pursuing robustness. Continuing our discussion at the beginning of this section, the general answer to Question 1 is that, $\mathcal{R}^\sigma$-risk consistency cannot guarantee its learnability as the gain in robustness may entail a bias. That is, in general, the Huber regression scheme \eqref{Huber_ERM} may not be mean regression calibrated. To address this problem and to learn the conditional mean function $f^\star$ through Huber regression, one needs to tune the scale parameter $\sigma$ to reduce the bias and learn in an adaptive way, as argued in the next section.       

\section{Huber Regression Is Asymptotically Mean Regression Calibrated}\label{sec::mean_calibration}

In this section, we shall show that, in a distribution-free setup, with properly selected scale parameter $\sigma$, Huber regression can be asymptotically mean regression calibrated, meaning that risk consistency implies the convergence of $f_{\mathbf z,\sigma}$ to  the conditional mean function $f^\star$ when $\sigma\rightarrow \infty$.

\subsection{Mean Regression Calibration Property}  

Recall that in the context of regression learning, one of the central concerns is the convergence of the learned empirical target function to the unknown truth function of interest, that is, the conditional mean function $f^\star$ in this study. While the distance between the empirical target function and $f^\star$ is not directly accessible, one settles for bounding the excess generalization error. %This is because for a learning machine $f$, its 
The underlying philosophy is that the generalization error of a learning machine can be approximated by using its empirical counterpart and the excess generalization error can be bounded via learning theory arguments. As mentioned in the introduction, the regression estimator is called \textit{mean regression calibrated} if the convergence of the excess generalization error towards $0$ implies the convergence of the empirical target function to the conditional mean function \cite{steinwart2008support}.

Translating into the context of Huber regression, one is concerned with whether the convergence of the excess generalization error $\mathcal{R}^\sigma(f_{\mathbf{z},\sigma})-\mathcal{R}^\sigma(f_\sigma)$ towards $0$ implies the convergence of $f_{\mathbf{z},\sigma}$ to $f^\star$. 
However, the numerical experiment and the counterexample in the preceding section suggest a negative answer and demonstrate that the desired mean regression calibration property may, in general, not be true. 
This conflicts with our intuition and common understanding that the Huber loss can serve as a robust alternate of the least squares loss when the scale parameter is chosen sufficiently large. To bypass this problem, in what follows, noticing our interest in learning the conditional mean function $f^\star$, we turn to investigate the relation between the convergence of $\mathcal{R}^\sigma(f_{\mathbf{z},\sigma})$ to $\mathcal{R}^\sigma(f^\star)$ and the convergence of $f_{\mathbf{z},\sigma}$ to $f^\star$. More specifically, we shall show that under mild conditions, the convergence of $\mathcal{R}^\sigma(f_{\mathbf{z},\sigma})$ to  $\mathcal{R}^\sigma(f^\star)$ does imply the convergence of $f_{\mathbf{z},\sigma}$ to $f^\star$ when $\sigma\rightarrow\infty$. This justifies the mean regression calibration property in an asymptotic sense.

\subsection{A Comparison Theorem}  

We now look into the asymptotic mean calibration property by establishing a comparison theorem. For regression estimators that are produced by empirical risk minimization schemes with convex loss functions, some efforts on investigating their mean regression calibration properties have been made in the literature; see e.g., \cite{steinwart2008support}. For Huber regression estimators, it was concluded that they are mean regression calibrated if the response variable is upper bounded or the conditional noise variables $\varepsilon|X$ admit symmetric probability density functions. Recall that one of the most prominent merits of Huber regression estimators lies in that they can perform mean regression in the absence of light-noise assumptions. In this sense, boundedness or symmetry constraints on the noise variable should be considered as stringent ones. In this study, we are seeking to assess the Huber regression estimator $f_{\mathbf{z},\sigma}$ and investigate its mean regression calibration properties without resorting to light-tail distributional assumptions on the conditional distribution or on the noise. To this end, we introduce the following weak moment condition.
\begin{assumption}\label{assump_moment}
	There exists a constant $\epsilon>0$ such that 
	$\mathbb{E}|Y|^{1+\epsilon}<+\infty$.
\end{assumption}	

The moment condition in Assumption \ref{assump_moment} is  rather weak in the sense that  it admits the case where the response variable $Y$  possesses infinite variance. The same comment condition also applies to the distributions of the conditional random variable $Y|X$ and the conditional noise variable $\varepsilon|X$ under the additive data generating model, implying that heavy-tailed noise is allowed. 

As discussed earlier, without further distributional assumptions on the noise variable, $f_{\mathbf{z},\sigma}$ is, in general, biased and its population version $f_\sigma$ may be different from $f^\star$ almost everywhere on $\mathcal{X}$. However, such a bias can be upper bounded and may decrease with the increase of the $\sigma$ values. Results in this regard are stated in the following theorem under the above $(1+\epsilon)$-moment condition. 

\begin{theorem}\label{theorem_comparison}
	Let $\sigma>\max\{2M, 1\}.$ Under Assumption \ref{assump_moment}, there exists an absolute constant ${c_\epsilon}>0$  independent of $\sigma$ such that for any measurable function $f:\mathcal{X}\rightarrow \mathbb{R}$ with $\|f\|_\infty\leq M$, we have 	
	\begin{align}\label{equ_compar_thm}
	\Big|\left[\mathcal{R}^\sigma(f)-\mathcal{R}^\sigma(f^\star)\right]- \|f-f^\star\|_{2,\rho}^2 \Big|\leq\cfrac{c_\epsilon}{\sigma^\epsilon}.
	\end{align}
\end{theorem}

Theorem  \ref{theorem_comparison} states that for any bounded measurable function $f$, under the $(1+\epsilon)$-moment conditions, the gap between   $\mathcal{R}^\sigma(f)-\mathcal{R}^\sigma(f^\star)$ and $\|f-f^\star\|_{2,\rho}^2$ is up to $\mathcal{O}(\sigma^{-\epsilon})$. 

Consequently, with a sufficiently large $\sigma$ value  or sufficiently light-tailed noise, this gap could be sufficiently small. As a special case, let us consider the presence of Gaussian or sub-Gaussian noise where the moment condition holds for arbitrarily large $\epsilon$ values. In this scenario, the gap between the above two quantities can be arbitrarily small. These findings remind us that in order to debias the Huber regression estimator, one may relate the $\sigma$ value  to the sample size $n$. In other words, from an asymptotic viewpoint, with diverging $\sigma$ values, according to Theorem \ref{theorem_comparison}, Huber regression is asymptotically mean regression calibrated. Following this spirit, we shall proceed with the assessment based on diverging $\sigma$ values by deriving their convergence rates to the conditional mean function $f^\star$.

\section{Relaxing the Bernstein Condition for Assessing Huber Regression}\label{sec::realxing_Bernstein}

From our previous discussions, in order to derive convergence rates for $f_{\mathbf{z},\sigma}$, one needs to bound the excess generalization error $\mathcal{R}^\sigma(f_{\mathbf{z},\sigma})-\mathcal{R}^\sigma(f^\star)$ which essentially requires us to deal with the following set of random variables
\begin{align*}
\mathcal{F}_\mathcal{H}:=\Big\{\xi\,|\,\xi(x,y)=\ell_\sigma(y-f(x))-\ell_\sigma(y-f^\star(x)),\,\,f\in\mathcal{H},\,\, (x,y)\in \mathcal{X}\times\mathcal{Y}\Big\}.
\end{align*}
The existing studies in learning theory remind us that  it is crucial to establish the so called \emph{Bernstein condition}, i.e., bounding the second moment of $\xi\in \mathcal{F}_\mathcal{H}$ by using its first moment. We will show that while the standard Bernstein condition does not hold, one can relax it to develop fast convergence rates for $\mathcal{R}^\sigma(f_{\mathbf{z},\sigma})-\mathcal{R}^\sigma(f^\star)$. Let us start with recapping the Bernstein condition in learning theory. 

\subsection{Bernstein Conditions in  Learning Theory}

Originally introduced in \cite{bartlett2006empirical} in the context of empirical risk minimization, the standard Bernstein condition can be restated as follows: a set $\mathcal{F}$ of  random variables is said to satisfy the $(\beta, B)$-Bernstein condition with $0<\beta\leq 1$  and $B>0$ if for any $f\in\mathcal{F}$, it holds that $\mathbb{E}f^2\leq B (\mathbb{E}(f))^\beta$. In other words, the second moment of the random variable (and so the variance) can be upper bounded by its first moment. Later, the standard Bernstein condition was generalized and extended into various other Bernstein-like conditions for analyzing learning algorithms in different contexts; see e.g., \cite{steinwart2007fast,van2015fast,chinot2019robust}. It turns out that the Bernstein condition and its variants play an important role in establishing fast convergence rates for learning algorithms of interest because they provide tight upper bounds for the variance of the random variables induced by the resulting estimators.  

In the context of Huber regression, a Bernstein-like condition is also desired in order to establish fast convergence rates for the excess generalization error $\mathcal{R}^\sigma(f_{\mathbf{z},\sigma})-\mathcal{R}^\sigma(f^\star)$. However, as shown in the preceding section, without further distributional restrictions to the noise variable, $f^\star$ may not be the optimal hypothesis that minimizes $\mathcal{R}^\sigma(f)$ over the measurable function space $\mathcal{M}$. Consequently, $\mathcal{R}^\sigma(f)-\mathcal{R}^\sigma(f^\star)$ is not necessarily positive. As a result, the usual Bernstein condition, namely, $\mathbb{E}\xi^2\leq B(\mathbb{E}\xi)^\beta$ for $\xi\in\mathcal{F}_\mathcal{H}$ with constants $B>0$ and $0<\beta\leq 1$, may not hold. This brings barriers to the development of fast convergence rates for $\mathcal{R}^\sigma(f)-\mathcal{R}^\sigma(f^\star)$.  
To circumvent this problem, in our study, we shall establish a relaxed Bernstein condition, which takes the form
\begin{align*} 
\mathbb{E}\xi^2\leq B(\mathbb{E}\xi)^\beta + g(\sigma),
\end{align*}
with $B>0$, $0<\beta\leq 1$, and $g$ a nonnegative function of $\sigma$. A motivating observation for establishing such a relaxed Bernstein condition is that the gap between $\mathbb{E}(\xi)$ and  $\|f-f^\star\|_{2,\rho}^2$ can be upper bounded by $\mathcal{O}(\sigma^{-\epsilon})$ for any $f\in\mathcal{H}$, as stated in Theorem \ref{theorem_comparison}.

\subsection{A Relaxed Bernstein Condition}

To establish our relaxed Bernstein condition, we  prove the following variance bound for $\xi\in\mathcal F_{\mathcal H}.$

\begin{theorem}\label{variance_estimate}
	Let Assumption \ref{assump_moment} hold and  $\sigma>\max\{2M,1\}$. For any measurable function $f: \mathcal{X}\rightarrow \mathbb{R}$ with $\|f\|_\infty\leq M$, the random variable  $\xi(x,y)=\ell_\sigma(y-f(x))-\ell_\sigma(y-f^\star(x))$ satisfies 
	\begin{align*}
	\mathbb{E}\xi^2\leq c_1\|f-f^\star\|_{2,\rho}^{\frac{2(\epsilon-1)_+}{\epsilon+1}} + c_2\sigma^{1-\epsilon},  \end{align*}
	where  $c_1$ and $c_2$ are absolute positive constants independent of $\sigma$ or $f$. 
\end{theorem}

Recall the results in Theorem \ref{theorem_comparison} that states 
$\|f-f^\star\|_{2,\rho}^2\leq \mathbb{E}\xi + c_\epsilon \sigma^{-\epsilon}.$
This in connection with Theorem \ref{variance_estimate} immediately yields the following relaxed Bernstein condition:
\begin{equation*}
	\mathbb{E}\xi^2\leq  
	c_1(\mathbb E \xi )^{\frac{(\epsilon-1)_+}{\epsilon+1}}
	+  c_1 (c_\epsilon \sigma^{-\epsilon}) ^{\frac{(\epsilon-1)_+}{\epsilon+1}} + c_2\sigma^{1-\epsilon}.    
\end{equation*}
This relaxed Bernstein condition will be crucial in establishing error bounds and fast exponential-type convergence rates for the Huber regression estimator $f_{\mathbf{z},\sigma}$.

\section{Assessing Huber Regression under Weak Moment Conditions}\label{sec::convergence_rates} 

In this section, we present fast exponential-type convergence rates for the Huber regression estimator under the $(1+\epsilon)$-moment condition in Assumption \ref{assump_moment}. Specifically, we are interested in bounding the $L^2_{\rho_{\mathcal{X}}}$-distance between $f_{\mathbf{z},\sigma}$ and $f^\star$.

To state the result, we introduce the following capacity assumption. For any $\eta>0$, let $\mathcal N(\mathcal H, \eta)$ denote the covering number of $\mathcal H$ by the balls of radius $\eta$ in $C(\mathcal X)$, that is, 
\begin{align*} 
\mathcal N(\mathcal H, \eta) = \min\left\{
k\in \mathbb N: \hbox{there exist } f_j\in \mathcal H, j=1,\ldots, k \hbox{ such that } \mathcal H\subset \bigcup_{j=1}^k B(f_j, \eta) \right \}, 
\end{align*} 
where $B(f_j, \eta)= \{f\in C(\mathcal X): \|f-f_j\|_\infty<\eta\}.$ 
Our capacity condition is stated as follows.

\begin{assumption}\label{complexity_assumption} 
	There exist positive constants $q$ and $c$ such that  
	$$\log\mathcal{N}(\mathcal{H},\eta)\leq c \eta^{-q},\,\, \forall\,\eta>0.$$
\end{assumption}
Generalization error bounds in terms of the covering number argument under Assumption \ref{complexity_assumption} is typical in 
statistical learning theory; see e.g., \cite{anthony2009neural, cucker2007learning, steinwart2008support} and references therein.

To state our results on the convergence rates, we introduce a new function $f_\mathcal{H}$, which is defined as 
\begin{align*}
f_\mathcal{H}=\arg\min_{f\in\mathcal{H}}\|f-f^\star\|_{2,\rho}^2.    
\end{align*}
The function $f_\mathcal{H}$ is the optimal function in $\mathcal{H}$ that one may expect in approximating the truth function $f^\star$. The distance $\|f_\mathcal{H}-f^\star\|_{2,\rho}^2$ can be regarded as the  approximation error when working with the hypothesis space $\mathcal{H}$ and so corresponds to the bias caused by the choice of the hypothesis space $\mathcal{H}$.

\begin{theorem}\label{thm::convergence_rates}
	Suppose that Assumptions \ref{assump_moment} and \ref{complexity_assumption} hold and let $\sigma>\max\{2M,1\}$.  Let $f_{\mathbf{z},\sigma}$ be produced by \eqref{Huber_ERM}. For any $0<\delta<1$, with probability at least $1-\delta$, it holds that 
	\begin{align*}
	\|f_{\mathbf{z},\sigma}-f^\star\|_{2,\rho}^2\lesssim \|f_\mathcal{H}-f^\star\|_{2,\rho}^2+\log(2/\delta)\Psi(n,\epsilon,\sigma),
	\end{align*}
	where 
	\begin{align*}
	\Psi(n,\epsilon,\sigma):=
	\begin{cases}
	\frac{1}{\sigma^\epsilon}+\frac{\sigma}{n^{1/(q+1)}},
	&\,\quad\hbox{if}\quad 0<\epsilon\leq 1,\\[1ex]
	\frac{1}{\sigma^\epsilon}+\left(\frac{\sigma^{q+\frac{2\epsilon}{1+\epsilon}}}{n}\right)^{1/(q+1)}, 
	&\,\quad \hbox{if}\quad \epsilon>1.
	\end{cases}
	\end{align*}
\end{theorem}

The proof of Theorem \ref{thm::convergence_rates} is based on a 
ratio probability inequality and standard learning theory argument \cite{devroye2013probabilistic,vapnik1998statistical, anthony2009neural,cucker2007learning,steinwart2008support}, where results established in Theorems \ref{theorem_comparison} and \ref{variance_estimate} play a crucial role. 

The error bound in Theorem \ref{thm::convergence_rates} involves three components: the approximation error due to the imperfect choice of the hypothesis space $\mathcal H,$ the inherent bias caused by the integrated parameter $\sigma,$ and the sample error. In practice, the hypothesis space could be chosen by structural risk minimization so that the approximation error decreases to a tolerably small level \cite{friedman2001elements,vapnik1998statistical}. The value of $\sigma$ affects both the inherent bias and the sample error. The best choice depends on the sample size, the moment condition,
and the capacity of the hypothesis space. To see this,
consider a special case when $f^\star\in\mathcal{H}$ so that the approximation error $\|f_\mathcal{H}-f^\star\|_{2,\rho}^2$ disappears. With properly chosen $\sigma$ values, we immediately obtain the following convergence rates. 
\begin{corollary}\label{corollary::convergence_rates}
	Under the assumptions of Theorem \ref{thm::convergence_rates}, let $f^\star\in\mathcal{H}$ and $\sigma$ be chosen as $\sigma=n^{\Phi(\epsilon,q)}$ with 
	\begin{align*}
	\Phi(\epsilon,q)=
	\begin{cases}
	\dfrac{1}{(1+\epsilon)(1+q)},\,\,\, & \hbox{ if}\,\,\,0<\epsilon\leq 1,\\[1.5em]
	\dfrac{1+\epsilon}{q(1+\epsilon)^2+\epsilon(\epsilon+3)},\,\,\,& \hbox{ if}\,\,\,\epsilon>1.
	\end{cases}
	\end{align*}
	For any $0<\delta<1$, with probability at least $1-\delta$, we have 
	\begin{align*}
	\|f_{\mathbf{z},\sigma}-f^\star\|_{2,\rho}^2\lesssim \log(2/\delta)n^{-\epsilon\Phi(\epsilon,q)}.
	\end{align*}
\end{corollary}

According to Corollary \ref{corollary::convergence_rates}, with properly chosen diverging $\sigma$ values, we obtain exponential-type convergence rates for $f_{\mathbf{z},\sigma}$. In particular, if the noise variable is bounded or sub-Gaussian and hence  the moment condition in Assumption \ref{assump_moment} holds for any $\epsilon>0,$
one can select an arbitrarily large $\epsilon$ to obtain convergence rates of order arbitrarily close to $O(n^{-\frac{1}{1+q}}).$ As a comparison, recall that for least square estimators convergence rates  of order $O(n^{-\frac 1{1+q/2}})$ can usually be established; see e.g., \cite{bauer2007regularization,steinwart2008support, anthony2009neural} and references therein. Moreover, note that with weaker moment conditions, i.e., smaller $\epsilon$ values, one gets slower convergence rates for $f_{\mathbf{z},\sigma}$, indicating increased sacrifice for robustness.  This coincides with our intuitive understanding of robust regression estimators.  
On the other hand, if $f^\star$ is smooth enough and one selects a smooth hypothesis space (such as the reproducing kernel Hilbert spaces induced by radial basis kernels or neural networks with smooth activate functions), then $q\to 0$ and the difference between Huber regression and least square method could be minimal, indicating less sacrifice necessary for learning smooth functions. Finally, we stress that we obtain exponential convergence rates even for the case when $0<\epsilon < 1$ where the distribution of the conditional variable $Y|X$ does not possess finite variance and a least square based estimator cannot even be defined.  This further explains the robustness of Huber regression estimators.

\section{Proofs of Theorems}\label{sec::proofs}
\subsection{Proof of Theorem \ref{theorem_comparison}}
	Let $f: \mathcal{X}\rightarrow \mathbb{R}$ with $\|f\|_\infty\leq M$ be a measurable function. For any $\sigma>\max\{2M, 1\}$, we denote the two events $\mathrm{I}_\mathrm{Y}$ and $\mathrm{II}_\mathrm{Y}$ as follows
	\begin{align*}
	\mathrm{I}_\mathrm{Y}:=\left\{y: |y|\geq\frac{\sigma}{2}\right\},
	\end{align*}
	and
	\begin{align*}
	\mathrm{II}_\mathrm{Y}:=\left\{y: |y|<\frac{\sigma}{2}\right\}.
	\end{align*}
	Noticing that 
	\begin{align*}
	\int_{\mathcal{X}}\int_{\mathcal{Y}}[(y-f(x))^2-(y-f^\star(x))^2]\mathrm{d}\rho(y|x)\mathrm{d}\rho_{\mathsmaller{\mathcal{X}}}(x)=\|f-f^\star\|_{2,\rho}^2,
	\end{align*}
	we have
	\begin{align*}
	&\Big|\left[\mathcal{R}^\sigma(f)-\mathcal{R}^\sigma(f^\star)\right]-\|f-f^\star\|_{2,\rho}^2\Big|\\[1ex]
	=\ &\Big|\int_{\mathcal{X}}\int_{\mathcal{Y}}\left[\ell_\sigma(y-f(x))-\ell_\sigma(y-f^\star(x))\right]-[(y-f(x))^2-(y-f^\star(x))^2]\mathrm{d}\rho(y|x)\mathrm{d}\rho_{\mathsmaller{\mathcal{X}}}(x)\Big|\\[1ex]
	\leq\ &\Big|\int_{\mathcal{X}}\int_{\mathrm{I}_\mathrm{Y}\bigcup\mathrm{II}_\mathrm{Y}}\left[\ell_\sigma(y-f(x))-\ell_\sigma(y-f^\star(x))\right]\mathrm{d}\rho(y|x)\mathrm{d}\rho_{\mathsmaller{\mathcal{X}}}(x)\Big|\\[1ex]
	&\ +\Big|\int_{\mathcal{X}}\int_{\mathrm{I}_\mathrm{Y}\bigcup\mathrm{II}_\mathrm{Y}}[(y-f(x))^2-(y-f^\star(x))^2]\mathrm{d}\rho(y|x)\mathrm{d}\rho_{\mathsmaller{\mathcal{X}}}(x)\Big|.
	\end{align*}
	For any $(x,y)\in \mathcal{X}\times\mathrm{II}_\mathrm{Y}$, since $\sigma>\max\{2M, 1\}$, we see that 
	\begin{align*}
	|y-f(x)|\leq |y|+\|f\|_{\infty}<\sigma,
	\end{align*}
	and
	\begin{align*}
	|y-f^\star(x)|\leq |y|+\|f^\star\|_\infty<\sigma.
	\end{align*}
	Consequently, for any $(x,y)\in\mathcal{X}\times \mathrm{II}_\mathrm{Y}$, we have  
	\begin{align*}
	\left[\ell_\sigma(y-f(x))-\ell_\sigma(y-f^\star(x))\right]-[(y-f(x))^2-(y-f^\star(x))^2]=0, 
	\end{align*}
	and hence
	\begin{align}
	&\Big|\left[\mathcal{R}^\sigma(f)-\mathcal{R}^\sigma(f^\star)\right]-\|f-f^\star\|_{2,\rho}^2\Big| \nonumber \\[1ex]
	\leq&\Big|\int_{\mathcal{X}}\int_{\mathrm{I}_\mathrm{Y}}\left[\ell_\sigma(y-f(x))-\ell_\sigma(y-f^\star(x))\right]\mathrm{d}\rho(y|x)\mathrm{d}\rho_{\mathsmaller{\mathcal{X}}}(x)\Big| \nonumber \\[1ex]
	&+\Big|\int_{\mathcal{X}}\int_{\mathrm{I}_\mathrm{Y}}[(y-f(x))^2-(y-f^\star(x))^2]\mathrm{d}\rho(y|x)\mathrm{d}\rho_{\mathsmaller{\mathcal{X}}}(x)\Big|. 
	\label{equ_two_terms}
	\end{align}
	Recall that the Huber loss \eqref{huber_loss} is Lipschitz continuous with Lipschitz constant $2\sigma$. The first term of the right-hand side of equation \eqref{equ_two_terms}  can be upper bounded as follows
	\begin{align*}
	&\Big|\int_{\mathcal{X}}\int_{\mathrm{I}_\mathrm{Y}}\left[\ell_\sigma(y-f(x))-\ell_\sigma(y-f^\star(x))\right]\mathrm{d}\rho(y|x)\mathrm{d}\rho_{\mathsmaller{\mathcal{X}}}(x)\Big|\\[1ex]
	\leq \ & 2\sigma\Big|\int_{\mathcal{X}}\int_{\mathrm{I}_\mathrm{Y}}|f(x)-f^\star(x)|\mathrm{d}\rho(y|x)\mathrm{d}\rho_{\mathsmaller{\mathcal{X}}}(x)\Big|\\[1ex]
	 \leq\  & 
	2\sigma\|f-f^\star\|_\infty \Pr(\mathrm{I}_\mathrm{Y}).
	\end{align*}
	The quantity $\Pr(\mathrm{I}_\mathrm{Y})$ can be bounded by applying Markov's inequality which yields
	\begin{align}\label{equ_prob}
	\Pr(\mathrm{I}_\mathrm{Y})\leq \cfrac{2^{1+\epsilon}\mathbb{E}\left(|Y|^{1+\epsilon}\right)}{\sigma^{1+\epsilon}}.
	\end{align} 
	Therefore, we have
	\begin{align}\label{equ_est1}
	\int_{\mathcal{X}}\int_{\mathrm{I}_\mathrm{Y}}\left[\ell_\sigma(y-f(x))-\ell_\sigma(y-f^\star(x))\right]\mathrm{d}\rho(y|x)\mathrm{d}\rho_{\mathsmaller{\mathcal{X}}}(x)\leq \cfrac{2^{2+\epsilon}\|f-f^\star\|_\infty \mathbb{E}\left(|Y|^{1+\epsilon}\right)}{\sigma^\epsilon}.
	\end{align}
	The second term in the right-hand side of equation \eqref{equ_two_terms} can be upper bounded as follows
	\begin{align*}
	&\quad\Big|\int_{\mathcal{X}}\int_{\mathrm{I}_\mathrm{Y}}[(y-f(x))^2-(y-f^\star(x))^2]\mathrm{d}\rho(y|x)\mathrm{d}\rho_{\mathsmaller{\mathcal{X}}}(x)\Big|\\
	&\leq \|f-f^\star\|_\infty\int_{\mathcal{X}}\int_{\mathrm{I}_\mathrm{Y}}|2y-f(x)-f^\star(x)|\mathrm{d}\rho(y|x)\mathrm{d}\rho_{\mathsmaller{\mathcal{X}}}(x)\\
	&\leq \|f-f^\star\|_\infty\int_{\mathcal{X}}\int_{\mathrm{I}_\mathrm{Y}}(2|y|+\|f^\star\|_\infty+\|f\|_\infty)\mathrm{d}\rho(y|x)\mathrm{d}\rho_{\mathsmaller{\mathcal{X}}}(x)\\
	&\leq \|f-f^\star\|_\infty \left(2\int_{\mathrm{I}_\mathrm{Y}}|y|\mathrm{d}\rho(y)+(\|f^\star\|_\infty+\|f\|_\infty)\Pr(\mathrm{I}_\mathrm{Y})\right).
	\end{align*}
	By applying H\"{o}lder inequality and recalling the estimate in \eqref{equ_prob}, we have
	\begin{align*}
	\int_{\mathrm{I}_\mathrm{Y}}|y|\mathrm{d}\rho(y)\leq \big(\Pr(\mathrm{I}_\mathrm{Y})\big)^{\frac{\epsilon}{1+\epsilon}}\big(\mathbb{E}(|Y|^{1+\epsilon})\big)^{\frac{1}{1+\epsilon}}\leq \cfrac{2^{\epsilon}\mathbb{E}\left(|Y|^{1+\epsilon}\right)}{\sigma^{\epsilon}}. 
	\end{align*}
	As a result, we  conclude that 
	\begin{align}
	&\quad\Big|\int_{\mathcal{X}}\int_{\mathrm{I}_\mathrm{Y}}[(y-f(x))^2-(y-f^\star(x))^2]\mathrm{d}\rho(y|x)\mathrm{d}\rho_{\mathsmaller{\mathcal{X}}}(x)\Big|
	\nonumber \\[1ex]
	&\leq  \cfrac{2^{1+\epsilon}\|f-f^\star\|_\infty\mathbb{E}\left(|Y|^{1+\epsilon}\right)}{\sigma^{\epsilon}}+\cfrac{2^{1+\epsilon}(\|f\|_\infty+\|f^\star\|_\infty)^2\mathbb{E}\left(|Y|^{1+\epsilon}\right)}{\sigma^{1+\epsilon}}. 
	\label{equ_est2}
	\end{align}
	From \eqref{equ_est1} and \eqref{equ_est2},   $\big|\left[\mathcal{R}^\sigma(f)-\mathcal{R}^\sigma(f^\star)\right]-\|f-f^\star\|_{2,\rho}^2\big|$ can be upper bounded by
	\begin{align*}
	\cfrac{(2^{2+\epsilon}+2^{1+\epsilon})\|f-f^\star\|_\infty \mathbb{E}\left(|Y|^{1+\epsilon}\right)}{\sigma^\epsilon}+\cfrac{2^{1+\epsilon}(\|f\|_\infty+\|f^\star\|_\infty)^2\mathbb{E}\left(|Y|^{1+\epsilon}\right)}{\sigma^{1+\epsilon}}.
	\end{align*}
	Therefore, the desired estimate \eqref{equ_compar_thm} holds with  $c_\epsilon=2^{3+\epsilon}(M+1)^2\mathbb{E}\left(|Y|^{1+\epsilon}\right)$. This completes the proof of Theorem \ref{theorem_comparison}.

\subsection{Proof of Theorem \ref{variance_estimate}}
	
Let $f: \mathcal{X}\rightarrow \mathbb{R}$ with $\|f\|_\infty\leq M$ be a measurable function. For any $\sigma>\max\{2M, 1\}$, we again consider the following two events  
\begin{align*}
\mathrm{I}_\mathrm{Y}:=\left\{y: |y|\geq\frac{\sigma}{2}\right\},
\end{align*}
and
\begin{align*}
\mathrm{II}_\mathrm{Y}:=\left\{y: |y|<\frac{\sigma}{2}\right\}.
\end{align*}
Based on the above notation, it is obvious to see the following decomposition 
\begin{align*}
\mathbb{E}\xi^2=&\int_{\mathcal{X}\times\mathcal{Y}}(\ell_\sigma(y-f(x))-\ell_\sigma(y-f^\star(x)))^2\mathrm{d}\rho(x,y)\\
=&\int_{\mathcal{X}\times\mathrm{I}_Y}(\ell_\sigma(y-f(x))-\ell_\sigma(y-f^\star(x)))^2\mathrm{d}\rho(x,y)\\
&+\int_{\mathcal{X}\times\mathrm{II}_Y}(\ell_\sigma(y-f(x))-\ell_\sigma(y-f^\star(x)))^2\mathrm{d}\rho(x,y)\\
:= & Q_1 +Q_2.
\end{align*}

The first term $Q_1$ can be easily bounded by applying the Lipschitz continuity property of the Huber loss \eqref{huber_loss} and Markov's inequality:  
\begin{align*}
Q_1 \leq & 4\sigma^2 \int_{\mathcal{X}\times\mathrm{I}_Y}\big(f(x)-f^\star(x))^2\mathrm{d}\rho(x,y)  % \\
\leq  16M^2\sigma^2 \Pr(\mathrm{I}_Y)\leq  16M^2\mathbb{E}|Y|^{1+\epsilon}\sigma^{1-\epsilon}.
\end{align*}
To bound the second term $Q_2,$ noticing that for any $(x,y)\in \mathcal{X}\times\mathrm{II}_\mathrm{Y}$, since $\sigma>\max\{2M, 1\}$, we have
\begin{align*}
|y-f(x)|\leq |y|+\|f\|_{\infty}<\sigma,
\end{align*}
and
\begin{align*}
|y-f^\star(x)|\leq |y|+\|f^\star\|_\infty<\sigma.
\end{align*}
By the definition of the Huber loss $\ell_\sigma$, for any $(x,y)\in \mathcal{X}\times\mathrm{II}_\mathrm{Y}$, we have  
\begin{align*}
\ell_\sigma(y-f(x))-\ell_\sigma(y-f^\star(x))= (y-f(x))^2-(y-f^\star(x))^2.
\end{align*}
Therefore,   
\begin{align*}
Q_2=&\int_{\mathcal{X}\times\mathrm{II}_Y}\big((y-f(x))^2-(y-f^\star(x))^2\big)^2\mathrm{d}\rho(x,y)\\[1ex]
=&\int_{\mathcal{X}}\int_{\mathrm{II}_Y}(f(x)-f^\star(x))^2(2y-f(x)-f^\star(x))^2\mathrm{d}\rho(y|x)\mathrm{d}\rho_{\mathsmaller{\mathcal{X}}}(x).
\end{align*}
If $\epsilon>1$, applying H\"{o}lder's inequality, we obtain
\begin{align*}
Q_2  &\leq  \int_{\mathcal{X}}\int_{\mathcal{Y}}(f(x)-f^\star(x))^2(2y-f(x)-f^\star(x))^2\mathrm{d}\rho(y|x)\mathrm{d}\rho_{\mathsmaller{\mathcal{X}}}(x) \\
&\leq \|f-f^\star\|_\infty^{\frac{4}{1+\epsilon}}\mathbb{E}\left(|f(x)-f^\star(x)|^{\frac{2\epsilon-2}{1+\epsilon}}(2|y|+2M)^2\right)\\
&\leq 64(M+1)^2 (\mathbb{E}|Y|^{1+\epsilon}+M^2+1)\|f-f^\star\|_{2,\rho}^{\frac{2\epsilon-2}{1+\epsilon}}.
\end{align*}
If $0<\epsilon\leq 1$, we have the following estimate 
\begin{align*}
Q_2 & \leq  48M^2\int_{\mathcal{X}}\int_{\mathrm{II}_Y}(|y|^{1+\epsilon}|y|^{1-\epsilon}+M^2) \mathrm{d}\rho(y|x)\mathrm{d}\rho_{\mathsmaller{\mathcal{X}}}(x) \\[1ex]
& \leq  48M^2(\mathbb{E}|Y|^{1+\epsilon}+M^{1+\epsilon})  \sigma^{1-\epsilon}.
\end{align*} 

Combining the above estimates for $Q_1$ and $Q_2$, we come to the conclusion that 
\begin{align*}
\mathbb{E}\xi^2\leq c_1\|f-f^\star\|_{2,\rho}^{\frac{2(\epsilon-1)_+}{\epsilon+1}} + c_2\sigma^{1-\epsilon},  
\end{align*}
with $c_1=64(M+1)^2 (\mathbb{E}|Y|^{1+\epsilon}+M^2+1)$ and $c_2=48M^2(\mathbb{E}|Y|^{1+\epsilon}+M^{1+\epsilon})+16M^2\mathbb{E}|Y|^{1+\epsilon}$. This completes the proof of Theorem \ref{variance_estimate}.

\subsection{Proof of Theorem \ref{thm::convergence_rates}}

We first prove a ratio inequality in Subsection \ref{subsec:ratio} 
which plays an important role in the proof of Theorem \ref{thm::convergence_rates}. The detailed proof will then be given in Subsection \ref{subsec:pf}. To proceed, for any measurable function $f:\mathcal{X}\rightarrow\mathbb{R}$, we denote
\begin{align*}
\mathcal{R}_\mathbf{z}^\sigma(f)=\frac{1}{n}\sum_{i=1}^n\ell_\sigma(y_i-f(x_i)),    
\end{align*}
and recall the notation 
\begin{align*}
f_{\mathcal{H},\sigma}=\arg\min_{f\in\mathcal{H}}\mathcal{R}^\sigma(f).    
\end{align*}

\subsubsection{A Ratio Inequality}
\label{subsec:ratio}

\begin{proposition}\label{ratio_inequality}
Let $\sigma>\max\{2M,1\}$. Under Assumptions \ref{assump_moment} and \ref{complexity_assumption}, for any $\gamma> \frac{c_\mathsmaller{\epsilon}}{\sigma^{\epsilon}}$, we have 
\begin{align*}
\Pr\left\{ \sup_{f\in\mathcal{H}}%\left\{
\frac{\big|[\mathcal{R}^{\sigma}(f)-\mathcal{R}^{\sigma}(f^\star)]-
	[\mathcal{R}_{\mathbf{z}}^\sigma(f)-\mathcal{R}_{\mathbf{z}}^\sigma(f^\star)]\big|}
{\sqrt{\mathcal{R}^{\sigma}(f)-\mathcal{R}^{\sigma}(f^\star)+2\gamma}} %\right\}
> 4\sqrt{\gamma} \right\} 
\le \mathcal{N}\left(\mathcal{H}, \frac{\gamma}{2\sigma} \right) e^{-\Theta(n,\gamma,\sigma)},
\end{align*} 
where  
\begin{align*}
\Theta(n,\gamma,\sigma)=
\begin{cases}
-\frac{n\gamma}{c_1^\prime \sigma}, & \quad\hbox{if}\,\,0<\epsilon \leq 1,\\[1ex]
-\frac{n\gamma}{c_2^\prime \sigma^{\frac{2\epsilon}{\epsilon+1}}}, & \quad\hbox{if}\,\, \epsilon > 1,
\end{cases}
\end{align*}
with $c_1^\prime$ and $c_2^\prime$ being two positive constants independent of $n$, $\gamma$, or $\sigma$ that will be explicitly specified in the proof.
\end{proposition}

To prove Proposition \ref{ratio_inequality}, we need the following Bernstein concentration inequality that is frequently employed in the literature of learning theory.

\begin{lemma}\label{oneside-bernstein}
	Let $\xi$ be a random variable on a probability space $\mathcal{Z}$
	with variance $\sigma_\star^2$ satisfying $|\xi-\mathbb{E}\xi|\leq
	M_{\xi}$ almost surely for some constant $M_{\xi}$ and for all
	$z\in\mathcal{Z}$. Then for all $\lambda>0$, 
	\begin{align*}
	\Pr 
	\left\{\frac{1}{n}\sum_{i=1}^n
	\xi(z_i)-\mathbb{E}\xi\geq \lambda\right\}\leq
	\exp\left\{-\frac{n
		\lambda^2}{2(\sigma_\star^2+\frac{1}{3}M_\xi\lambda)}\right\}.
	\end{align*}
\end{lemma}

\begin{proof}[Proof of Proposition \ref{ratio_inequality}.]
Recall that $\mathcal{F}_\mathcal{H}$ denotes the following set of  random variables 
\begin{align*}
\mathcal{F}_\mathcal{H}=\Big\{\xi\,|\,\xi(x,y)=\ell_\sigma(y-f(x))-\ell_\sigma(y-f^\star(x)),\,\,f\in\mathcal{H},\,\, (x,y)\in \mathcal{X}\times\mathcal{Y}\Big\}.
\end{align*}
For each $\xi\in\mathcal F_{\mathcal H}$, by the fact that the Huber loss \eqref{huber_loss}  is Lipschitz continuous with Lipschitz constant $2\sigma$, we have
\begin{align*}
\|\xi\|_\infty\leq
2\sigma \|f-f^\star\|_\infty\leq 4M \sigma \,\,\,\hbox{and}\,\,\, \|\xi-\mathbb{E}\xi\|_\infty\leq
4\sigma \|f-f^\star\|_\infty\leq 8M \sigma. 
\end{align*}
According to Theorem \ref{variance_estimate}, we know that 
\begin{align*}
\mathbb{E}\xi^2\leq c_1\|f-f^\star\|_{2,\rho}^{\frac{2(\epsilon-1)_+}{\epsilon+1}} + c_2\sigma^{1-\epsilon}.    
\end{align*}

By Assumption \ref{complexity_assumption}, we know that there exist a finite positive integer $J=\mathcal{N}(\mathcal{H},\frac{\gamma}{2\sigma})$ and $\{f_j\}_{j=1}^J\subset \mathcal{H}$ such that $B(f_j, \eta), j=1,\ldots,J$ form a $\frac{\gamma}{2\sigma}$-cover of $\mathcal{H}$. We next show that for each $j=1,\cdots,J$, it holds that 
\begin{align}\label{temp_ratio_inequ}
&\Pr 
\left\{\frac{\big|[\mathcal{R}^{\sigma}(f_j)-\mathcal{R}^{\sigma}(f^\star)]-
	[\mathcal{R}_{\mathbf{z}}^\sigma(f_j)-\mathcal{R}_{\mathbf{z}}^\sigma(f^\star)]\big|}
{\sqrt{\mathcal{R}^{\sigma}(f_j)-\mathcal{R}^{\sigma}(f^\star)+2\gamma}}>\sqrt{\gamma}\right \}\leq   e^{-\Theta(n,\gamma, \sigma)}
\end{align}
for $\gamma > \frac{c_\mathsmaller{\epsilon}}{\sigma^{\epsilon}}.$
To see this, we apply the Bernstein inequality in Lemma \ref{oneside-bernstein} to the following random variables
\begin{align*}
\xi_j(x,y)=\ell_\sigma(y-f_j(x))-
\ell_\sigma(y-f^\star(x)),\, (x,y)\in\mathcal{X}\times\mathcal{Y}, \,j=1,\ldots, J,
\end{align*}
and obtain
\begin{align}
&\Pr
\left\{\big|[\mathcal{R}^{\sigma}(f_j)-\mathcal{R}^{\sigma}(f^\star)]-
[\mathcal{R}_{\mathbf{z}}^\sigma(f_j)-\mathcal{R}_{\mathbf{z}}^\sigma(f^\star)]\big|
>\sqrt{\gamma}\sqrt{\mathcal{R}^{\sigma}(f_j)-\mathcal{R}^{\sigma}(f^\star)+2\gamma}\right\} \nonumber \\[0.5em]
\leq\ & \Pr
\left\{\big|[\mathcal{R}^{\sigma}(f_j)-\mathcal{R}^{\sigma}(f^\star)]-
[\mathcal{R}_{\mathbf{z}}^\sigma(f_j)-\mathcal{R}_{\mathbf{z}}^\sigma(f^\star)]\big|
>\mu_j\sqrt{\gamma}\right\} \nonumber \\[0.5em]
\leq\ &\exp\left\{-\frac{n\gamma\mu_j^2}{(8M/3+c_1+c_2)\left(\sqrt{\gamma}\mu_j\sigma+
	\sigma^{1-\epsilon}+\|f_j-f^\star\|_{2,\rho}^{\frac{2(\epsilon-1)_+}{\epsilon+1}}\right)}\right\},
\label{temp_raito_ii}
\end{align}
where $\mu_j^2:=\mathcal{R}^{\sigma}(f_j)-\mathcal{R}^{\sigma}(f^\star)+2\gamma$. 
Since $\gamma> \frac{  c_\mathsmaller{\epsilon}}{\sigma^{\epsilon}},$ by Theorem \ref{theorem_comparison}, we have for $j=1,\ldots,J$, 
\begin{align}
\mu_j^2&=\mathcal{R}^{\sigma}(f_j)-\mathcal{R}^{\sigma}(f^\star)+2\gamma \nonumber \\
& >
\mathcal{R}^{\sigma}(f_j)-\mathcal{R}^{\sigma}(f^\star)+c_\mathsmaller{\epsilon}\sigma^{-\epsilon}+\gamma \nonumber \\
& \ge  \|f_j-f^\star\|_{2,\rho}^2+\gamma\geq
\gamma.
\label{mu_j}
\end{align}
We proceed with the proof by considering the two cases when $0<\epsilon\leq 1$ and when $\epsilon>1$. If $0<\epsilon\leq 1$, by the assumption $\sigma>1$ and \eqref{mu_j}, we have
\begin{align*}
\hspace{-3cm}& \frac{n\gamma\mu_j^2}{(8M/3+c_1+c_2)\left(\sqrt{\gamma}\mu_j\sigma+
	\sigma^{1-\epsilon}+\|f_j-f^\star\|_{2,\rho}^{\frac{2(\epsilon-1)_+}{\epsilon+1}}\right)}\\
 >\  &   \frac{n\gamma\mu_j^2}{2(8M/3+c_1+c_2)\left(\sqrt{\gamma}\mu_j\sigma+
	\sigma^{1-\epsilon}\right)} > \frac{n\gamma}{c_1^\prime \sigma},
\end{align*}
where $c_1^\prime=2(1+c_\epsilon^{-1})(8M/3+c_1+c_2).$
If $\epsilon>1$, note that \eqref{mu_j} implies
 $\mu_j^2 > \|f_j-f^\star\|_{2,\rho}^2$ and  $\mu_j^2 > \gamma > c_\epsilon\sigma^{-\epsilon}$.
We have 
\begin{align*}
\hspace{-2cm}& \frac{n\gamma\mu_j^2}{(8M/3+c_1+c_2)\left(\sqrt{\gamma}\mu_j\sigma+
	\sigma^{1-\epsilon}+\|f_j-f^\star\|_{2,\rho}^{\frac{2(\epsilon-1)_+}{\epsilon+1}}\right)}\\
=\  &   \frac{n\gamma\mu_j^2}{(8M/3+c_1+c_2)\left(\sqrt{\gamma}\mu_j\sigma+
	\sigma^{1-\epsilon}+\|f_j-f^\star\|_{2,\rho}^{\frac{2(\epsilon-1)}{\epsilon+1}}\right)}\\
>\  & \frac{n\gamma\mu_j^2}{(8M/3+c_1+c_2)\left(\sqrt{\gamma}\mu_j\sigma+
	\sigma^{1-\epsilon}+\mu_j^{\frac{2(\epsilon-1)}{\epsilon+1}}\right)} >  \cfrac{n\gamma}{c_2^\prime\sigma^{\frac{2\epsilon}{\epsilon+1}}},
\end{align*}
where $c_2^\prime= (8M/3+c_1+c_2)\left(1+c_\epsilon^{-1}+c_\epsilon^{-\frac{2}{\epsilon+1}}\right)$. Combining the above estimates for two cases and recalling \eqref{temp_raito_ii}, we thus have proved the result in \eqref{temp_ratio_inequ}.

Denote the two events $A$ and, respectively, as
\begin{align*} 
A = \left\{ \sup_{f\in\mathcal{H}}
\frac{\big|[\mathcal{R}^{\sigma}(f)-\mathcal{R}^{\sigma}(f^\star)]-
	[\mathcal{R}_{\mathbf{z}}^\sigma(f)-\mathcal{R}_{\mathbf{z}}^\sigma(f^\star)]\big|}
{\sqrt{\mathcal{R}^{\sigma}(f)-\mathcal{R}^{\sigma}(f^\star)+2\gamma}} \le 
4\sqrt{\gamma } \right\} 
\end{align*} 
and 
\begin{align*} 
B = 
 \bigcap _{j=1}^J \left\{\frac{\big|[\mathcal{R}^{\sigma}(f_j)-\mathcal{R}^{\sigma}(f^\star)]-
	[\mathcal{R}_{\mathbf{z}}^\sigma(f_j)-\mathcal{R}_{\mathbf{z}}^\sigma(f^\star)]\big|}
{\sqrt{\mathcal{R}^{\sigma}(f_j)-\mathcal{R}^{\sigma}(f^\star)+2\gamma}} \le \sqrt{\gamma}\right\}.
\end{align*}
We next prove $B\subset A.$ 
To this end, we assume that the event $B$ occurs, that is, for all $j=1, \ldots, J$, 
\begin{align}
\big|[\mathcal{R}^{\sigma}(f_j)-\mathcal{R}^{\sigma}(f^\star)]-
[\mathcal{R}_{\mathbf{z}}^\sigma(f_j)-\mathcal{R}_\mathbf{z}^{\sigma}(f^\star)]\big|
\le \sqrt{\gamma \left(\mathcal{R}^{\sigma}(f_j)-\mathcal{R}^{\sigma}(f^\star)+2\gamma\right)}.
\label{eq:fj}
\end{align}
Recall $B(f_j, \eta), j=1, \ldots, J,$ is a  $\frac{\gamma}{2\sigma}$-cover of $\mathcal{H}$. For every $f\in\mathcal{H}$, there exists $f_j\in\mathcal{H}$ such that
$\|f-f_j\|_\infty\leq  \frac{ \gamma}{2\sigma}$.
Since $\gamma > {c_\epsilon}{\sigma^{-\epsilon}},$ we have
\begin{align}\label{temp_3}
\mathcal{R}^\sigma(f)-\mathcal{R}^\sigma(f^\star)+2\gamma\geq \|f-f^\star\|_{2,\rho}^2-c_\mathsmaller{\epsilon}\sigma^{-\epsilon}+2\gamma\geq \gamma.
\end{align} 
Therefore, by the Lipschitz continuity of the Huber loss, for each $j=1,\dots,J$, we have
\begin{align}\label{temp_1}
|\mathcal{R}^\sigma(f)-\mathcal{R}^\sigma(f_j)|\leq 2\sigma\|f-f_j\|_\infty\leq \gamma 
\le \sqrt{\gamma \left(\mathcal{R}^{\sigma}(f)-\mathcal{R}^{\sigma}(f^\star)+2\gamma\right)}, 
\end{align}
and 
\begin{align}\label{temp_2}
|\mathcal{R}_{\mathbf{z}}^\sigma(f)-\mathcal{R}_{\mathbf{z}}^\sigma(f_j)|\leq 2\sigma\|f-f_j\|_\infty\leq \gamma
\sqrt{\gamma \left(\mathcal{R}^{\sigma}(f)-\mathcal{R}^{\sigma}(f^\star)+2\gamma\right)}. 
\end{align}
By \eqref{temp_1} and \eqref{temp_3}, we also have 
\begin{align}
\mathcal{R}^{\sigma}(f_j)-\mathcal{R}^{\sigma}(f^\star)+2\gamma &=
\Big (\mathcal{R}^{\sigma}(f_j)-\mathcal{R}^{\sigma}(f) \Big) +
\mathcal{R}^{\sigma}(f)-\mathcal{R}^{\sigma}(f^\star)+2\gamma \nonumber \\
&\leq\sqrt{\gamma\left(\mathcal{R}^{\sigma}(f)-\mathcal{R}^{\sigma}(f^\star)+2\gamma\right)}
+\mathcal{R}^{\sigma}(f)-\mathcal{R}^{\sigma}(f^\star)+2\gamma \nonumber \\
&\leq 2(\mathcal{R}^{\sigma}(f)-\mathcal{R}^{\sigma}(f^\star)+2\gamma).
\label {eq:temp_4}
\end{align}
Combining the estimates \eqref{temp_1}, \eqref{temp_2}, 
\eqref{eq:temp_4} with the assumption \eqref{eq:fj}, we obtain 
\begin{align}
& \big|[\mathcal{R}^{\sigma}(f)-\mathcal{R}^{\sigma}(f^\star)]-
[\mathcal{R}_{\mathbf{z}}^\sigma(f)-\mathcal{R}_\mathbf{z}^{\sigma}(f^\star)]\big| \nonumber \\[1ex] 
\le\ & | \mathcal{R}^{\sigma}(f) - \mathcal{R}^{\sigma}(f_j) | 
+ \big |\mathcal{R}^{\sigma}(f_j)   -\mathcal{R}^{\sigma}(f^\star)]-
[\mathcal{R}_{\mathbf{z}}^\sigma(f_j)-\mathcal{R}_\mathbf{z}^{\sigma}(f^\star)]\big|
+ | \mathcal{R}_{\mathbf{z}}^\sigma(f_j) -\mathcal{R}_{\mathbf{z}}^\sigma(f) | \nonumber \\[1ex]
\le\ & 
2 \sqrt{\gamma \left(\mathcal{R}^{\sigma}(f)-\mathcal{R}^{\sigma}(f^\star)+2\gamma\right)} +\sqrt{\gamma \left(\mathcal{R}^{\sigma}(f_j)-\mathcal{R}^{\sigma}(f^\star)+2\gamma\right)} \nonumber \\[1ex] 
\le \ & 4\sqrt{\gamma \left(\mathcal{R}^{\sigma}(f)-\mathcal{R}^{\sigma}(f^\star)+2\gamma\right)}.
\label{eq:f}
\end{align}
Since \eqref{eq:f} holds for every $f\in\mathcal H,$ we have proved $B\subset A$ or equivalently $A^c\subset B^c.$ This together with \eqref{temp_ratio_inequ}
 leads to 
 \begin{align*}
&\Pr
\left\{\sup_{f\in\mathcal{H}}
\frac{\big|[\mathcal{R}^{\sigma}(f)-\mathcal{R}^{\sigma}(f^\star)]-
	[\mathcal{R}_{\mathbf{z}}^\sigma(f)-\mathcal{R}_{\mathbf{z}}^\sigma(f^\star)]\big|}
{\sqrt{\mathcal{R}^{\sigma}(f)-\mathcal{R}^{\sigma}(f^\star)+2\gamma}}>
4\sqrt{\gamma } \right\}\\[0.5em]
= \ & \Pr(A^c) \le \Pr(B^c) \\[0.5em]  
\le \ &
\sum_{j=1}^J \Pr 
\left\{\frac{\big|[\mathcal{R}^{\sigma}(f_j)-\mathcal{R}^{\sigma}(f^\star)]-
	[\mathcal{R}_{\mathbf{z}}^\sigma(f_j)-\mathcal{R}_{\mathbf{z}}^\sigma(f^\star)]\big|}
{\sqrt{\mathcal{R}^{\sigma}(f_j)-\mathcal{R}^{\sigma}(f^\star)+2\gamma}}>\sqrt{\gamma}\right\}\\[0.5em]
\leq\ & 
\mathcal{N}\left(\mathcal{H}, \gamma\sigma^{-1}/2 \right) e^{-\Theta(n,\gamma,\sigma)}.
\end{align*}
This completes the proof of Proposition \ref{ratio_inequality}. 
\end{proof}

\subsubsection{Proof of the Theorem}\label{subsec:pf}

We first prove that for any $0<\delta<1$, with probability at least $1-\delta/2$, we have
\begin{align*}
[\mathcal{R}^{\sigma}(f_{\mathbf{z},\sigma})-\mathcal{R}^{\sigma}(f^\star)]-
[\mathcal{R}_{\mathbf{z}}^\sigma(f_{\mathbf{z},\sigma})-\mathcal{R}_{\mathbf{z}}^\sigma(f^\star)]- \frac{1}{2}[\mathcal{R}^{\sigma}(f_{\mathbf{z},\sigma})-\mathcal{R}^{\sigma}(f^\star)]\leq 9\gamma_0,
\end{align*}
where $\gamma_0$ is given by
\begin{align}
\gamma_0:= 
\begin{cases}
\mathcal{O}\left(\frac{1}{\sigma^{\epsilon}}+\log\left(\frac{2}{\delta}\right)\frac{\sigma}{n^{1/(q+1)}}\right),&\,\,\hbox{if}\,\,0<\epsilon\leq 1,\\[1em]
\mathcal{O}\left(\frac{1}{\sigma^{\epsilon}}+\log\left(\frac{2}{\delta}\right)\left(\frac{\sigma^{q+\frac{2\epsilon}{1+\epsilon}}}{n}\right)^{1/(q+1)}\right),&\,\,\hbox{if}\,\,\epsilon > 1.
\end{cases}
\label{eq:gamma0} 
\end{align} 
Note that Proposition \ref{ratio_inequality} implies that, for any $\gamma> c_\epsilon \sigma^{-\epsilon}$,
\begin{align} 
\sup_{f\in\mathcal{H}}
\frac{\big|[\mathcal{R}^{\sigma}(f)-\mathcal{R}^{\sigma}(f^\star)]-
	[\mathcal{R}_{\mathbf{z}}^\sigma(f)-\mathcal{R}_{\mathbf{z}}^\sigma(f^\star)]\big|}
{\sqrt{\mathcal{R}^{\sigma}(f)-\mathcal{R}^{\sigma}(f^\star)+2\gamma}}< 
4\sqrt{\gamma }
\label{eq:allf}
\end{align} 
holds with probability at least  $$1- \mathcal{N}\left(\mathcal{H}, \frac{\gamma}{2 \sigma} \right)e^{-\Theta(n, \gamma,\sigma)}.$$ 
We know from Assumption \ref{complexity_assumption} that
\begin{align*}
\mathcal{N}\left(\mathcal{H}, \frac{\gamma}{2\sigma} \right) \lesssim
\exp\left\{2^q c\sigma^q/\gamma^q\right\}.
\end{align*}
Consequently, the event \eqref{eq:allf} holds 
with probability at least 
$1-\exp\{2^q c\sigma^q \gamma^{-q} - \Theta(n,\gamma,\sigma)\}$. 
For any $0<\delta<1$, let 
$$\exp\{2^q c\sigma^q \gamma^{-q} - \Theta(n,\gamma,\sigma)\}= \delta/2,$$ or equivalently 
\begin{align*}
2^q c \sigma^q\gamma^{-q}-\Theta(n,\gamma,\sigma)=\log(\delta/2).
\end{align*}
The equation has a unique positive solution $\gamma^*$ satisfying 
\begin{align*}
\gamma^\star\lesssim 
\begin{cases} 
\log\left(\frac{2}{\delta}\right)\left(\frac{\sigma}{n^{1/(q+1)}}\right), & \hbox{ if } \epsilon\le 1, \\[1ex]
\log\left(\frac{2}{\delta}\right)\left(\frac{\sigma^{q+\frac{2\epsilon}{1+\epsilon}}}{n}\right)^{1/(q+1)}, 
& \hbox{ if } \epsilon>1.
\end{cases}
\end{align*}
Choose $\gamma_0=c_\epsilon\sigma^{-\epsilon}+\gamma^\star.$ Then $\gamma_0$ satisfies the condition \eqref{eq:gamma0} and for any $0<\delta<1$, with probability at least $1-\delta/2$, it holds that
\begin{align*} 
\sup_{f\in\mathcal{H}}
\frac{\big|[\mathcal{R}^{\sigma}(f)-\mathcal{R}^{\sigma}(f^\star)]-
	[\mathcal{R}_{\mathbf{z}}^\sigma(f)-\mathcal{R}_{\mathbf{z}}^\sigma(f^\star)]\big|}
{\sqrt{\mathcal{R}^{\sigma}(f)-\mathcal{R}^{\sigma}(f^\star)+2\gamma_0}} \le 
4\sqrt{\gamma_0},
\end{align*}
which immediately yields
\begin{align}\label{sample_error_I}
[\mathcal{R}^{\sigma}(f_{\mathbf{z},\sigma})-\mathcal{R}^{\sigma}(f^\star)]-
[\mathcal{R}_{\mathbf{z}}^\sigma(f_{\mathbf{z},\sigma})-\mathcal{R}_{\mathbf{z}}^\sigma(f^\star)]& \leq
4\sqrt{\gamma_0}\sqrt{\mathcal{R}^{\sigma}(f_{\mathbf{z},\sigma})-\mathcal{R}^{\sigma}(f^\star)+2\gamma_0} \nonumber \\[1ex] 
& \le \frac 12 \Big ( \mathcal{R}^{\sigma}(f_{\mathbf{z},\sigma})-\mathcal{R}^{\sigma}(f^\star) \Big) + 9 \gamma_0.
\end{align}
By a similar procedure 
we can prove that for any $0<\delta<1$, with probability at least $1-\delta/2$, it holds that
\begin{align*}
[\mathcal{R}_{\mathbf{z}}^\sigma(f_{\mathcal{H},\sigma})-\mathcal{R}_{\mathbf{z}}^\sigma(f^\star)]-[\mathcal{R}^\sigma(f_{\mathcal{H},\sigma})-\mathcal{R}^\sigma(f^\star)]& \le \frac{1}{2}\Big( \mathcal{R}^{\sigma}(f_{\mathcal{H},\sigma})-\mathcal{R}^{\sigma}(f^\star)\Big) +  9\gamma_0. 
\end{align*}
This in connection with the fact that $\mathcal{R}^\sigma(f_{\mathcal{H},\sigma})\leq \mathcal{R}^\sigma(f_\mathcal{H})$ and Theorem \ref{theorem_comparison} implies that
for any $0<\delta<1$, with probability at least $1-\delta/2$, we have
\begin{align}
[\mathcal{R}_{\mathbf{z}}^\sigma(f_{\mathcal{H},\sigma})-\mathcal{R}_{\mathbf{z}}^\sigma(f^\star)]-[\mathcal{R}^\sigma(f_{\mathcal{H},\sigma})-\mathcal{R}^\sigma(f^\star)]& \le \frac{1}{2}\| f_{\mathcal{H}}- f^\star\|_{2,\rho}^2  +  10\gamma_0. 
\label{sample_error_II}
\end{align}

Combining the two estimates in \eqref{sample_error_I} and \eqref{sample_error_II}, we come to the conclusion that for any $0<\delta<1$, with probability at least $1-\delta$, it holds that 
\begin{align}
& [\mathcal{R}^\sigma(f_{\mathbf{z},\sigma})-\mathcal{R}^\sigma(f_{\mathcal{H},\sigma})]-[\mathcal{R}_{\mathbf{z}}^\sigma(f_{\mathbf{z},\sigma})-\mathcal{R}_{\mathbf{z}}^\sigma(f_{\mathcal{H},\sigma})]  \nonumber \\[1ex]
\leq & \frac{1}{2}\big[\mathcal{R}^\sigma(f_{\mathbf{z},\sigma})-\mathcal{R}^\sigma(f^\star)\big]+ \frac{1}{2}\|f_\mathcal{H}-f^\star\|_{2,\rho}^2+19 \gamma_0.
\label{eq:1/2bound}
\end{align}

On the other hand, from the definitions of $f_{\mathcal{H},\sigma}$, $f_\mathcal{H}$, and $f_{\mathbf{z},\sigma}$, we have 
\begin{align*}
&\mathcal{R}^\sigma(f_{\mathbf{z},\sigma})-\mathcal{R}^\sigma(f^\star)\\[1ex]
=&[\mathcal{R}^\sigma(f_{\mathbf{z},\sigma})-\mathcal{R}^\sigma(f_{\mathcal{H},\sigma})]+[\mathcal{R}^\sigma(f_{\mathcal{H},\sigma})-\mathcal{R}^\sigma(f^\star)]\\[1ex]
\leq& [\mathcal{R}^\sigma(f_{\mathbf{z},\sigma})-\mathcal{R}^\sigma(f_{\mathcal{H},\sigma})]-[\mathcal{R}_{\mathbf{z}}^\sigma(f_{\mathbf{z},\sigma})-\mathcal{R}_{\mathbf{z}}^\sigma(f_{\mathcal{H},\sigma})]+\mathcal{R}^\sigma(f_\mathcal{H})-\mathcal{R}^\sigma(f^\star)\\[1ex]
\leq& [\mathcal{R}^\sigma(f_{\mathbf{z},\sigma})-\mathcal{R}^\sigma(f_{\mathcal{H},\sigma})]-[\mathcal{R}_{\mathbf{z}}^\sigma(f_{\mathbf{z},\sigma})-\mathcal{R}_{\mathbf{z}}^\sigma(f_{\mathcal{H},\sigma})]+ \|f_\mathcal{H}-f^\star\|_{2,\rho}^2+c_\epsilon\sigma^{-\epsilon},
\end{align*}
where the first inequality is due to the following two facts
\begin{align*}
\mathcal{R}_{\mathbf{z}}^\sigma(f_{\mathbf{z},\sigma})\leq \mathcal{R}_{\mathbf{z}}^\sigma(f_{\mathcal{H},\sigma}),\quad\hbox{and}\quad\mathcal{R}^\sigma(f_{\mathcal{H},\sigma})\leq \mathcal{R}^\sigma(f_\mathcal{H}).
\end{align*}
By \eqref{eq:1/2bound}, we know that for any $0<\delta<1$, with probability at least $1-\delta$, it holds that 
\begin{align*}
    \mathcal{R}^\sigma(f_{\mathbf{z},\sigma})-\mathcal{R}^\sigma(f^\star) \le \frac 12 \Big( \mathcal{R}^\sigma(f_{\mathbf{z},\sigma})-\mathcal{R}^\sigma(f^\star)\Big) + \frac 32 \|f_{\mathcal H} - f^*\|_{2,\rho}^2 + 20 \gamma_0,
\end{align*}
which implies that 
\begin{align*}
    \mathcal{R}^\sigma(f_{\mathbf{z},\sigma})-\mathcal{R}^\sigma(f^\star) \leq  3 \|f_{\mathcal H} - f^*\|_{2,\rho}^2 + 40 \gamma_0
\end{align*}
holds with probability at least $1-\delta$. By Theorem \ref{theorem_comparison} again, we conclude that for any $0<\delta<1$, with probability at least $1-\delta$, it holds that  
\begin{align*}
\|f_{\mathbf{z},\sigma}-f^\star\|_{2,\rho}^2\lesssim \|f_\mathcal{H}-f^\star\|_{2,\rho}^2+\gamma_0.
\end{align*}
Recalling the definition of $\Psi$ and noticing that $\gamma_0 \lesssim \log(\delta/2 ) \Psi$, we  complete the proof of Theorem \ref{thm::convergence_rates}. 

\section{Concluding Remarks}\label{sec:conclusion}
In this paper, we studied the Huber regression problem by investigating the empirical risk minimization scheme induced by the Huber loss. In a statistical learning setup, our study answered the four fundamental questions raised in the introduction: the $\mathcal{R}^\sigma$-risk consistency is insufficient in ensuring their convergence to the mean regression function; the scale parameter $\sigma$ plays a trade-off role in bias and learnability; fast exponential-type convergence rates can be established under $(1+\epsilon)$-moment conditions ($\epsilon>0$) by relaxing the standard Bernstein condition and allowing some additional small bias term; the merit of Huber regression in terms of the robustness can be reflected by its learnability under the $(1+\epsilon)$-moment conditions which are considered to be weak conditions in that heavy-tailed noise can be accommodated in regression problems. Moreover, it was shown that with higher moment conditions being imposed, one can obtain faster convergence rates. In the above senses, we conducted a complete and systematic statistical learning assessment of Huber regression estimators.  

We remark that in the present study a general hypothesis space $\mathcal{H}$ is considered. In practice, the implementation of learning with Huber regression requires to specify a particular hypothesis space. It can be a reproducing kernel Hilbert space, a neural network, or other families of functions. Functions in such a hypothesis space are generally not uniformly bounded. Regularization could be used to restrict the searching region of the Huber regression scheme and consequently controls the capacity of the working hypothesis space. The techniques developed in this study may still be applicable to assessing the regularized Huber regression schemes. Additionally, the development of these techniques for assessing Huber regression estimators may also shed light on the analysis of other robust regression schemes.

\section*{Acknowledgement}
This work was partially supported by the Simons Foundation Collaboration Grant \#572064 (YF) and \#712916 (QW). The two authors made equal contributions to this paper and are listed alphabetically. 

\bibliographystyle{abbrv}
\bibliography{FW2020a}
\end{document}